\documentclass[11pt]{article}

\usepackage{amsmath,amscd,amssymb,amsmath,amsthm}
\usepackage{geometry}
\usepackage{enumitem}
\usepackage[utf8]{inputenc} 
\usepackage{authblk} 

\title{About the integrability of the Rapcsák equation}

\author{T. Milkovszki and Z. Muzsnay} \affil{\small Institute of Mathematics,
  University of Debrecen,\authorcr H-4032 Debrecen, Egyetem t\'er 1,
  Hungary\authorcr {\it E-mail}: {\tt
    {}[milkovszki,muzsnay]@science.unideb.hu}\authorcr \hspace{1cm}\authorcr }

\allowdisplaybreaks

\theoremstyle{plain} 
\newtheorem{theorem}{Theorem}[section]
\newtheorem{proposition}[theorem]{Proposition}

\newtheorem{corollary}[theorem]{Corollary}
\newtheorem{lemma}[theorem]{Lemma}
\theoremstyle{definition}
\newtheorem{definition}[theorem]{Definition}
\newtheorem{remark}[theorem]{Remark}

\theoremstyle{remark}

\newcommand{\R}{\mbox{$\mathbb R$}} 
\newcommand{\ts}{\textsuperscript}

\def\TM{\mathcal{T}M}

\def\u#1{\underline{#1}}
\def\Ker{\mathrm{Ker \,}}
\def\Im{\mathrm{Im \,}}

\def\rank{\mathrm{rank \,}}
\def\RTM{\mathbb R_{TM}}
\def\L{\mathcal L}

\def\P{\mathcal P}

\begin{document}

\maketitle

\begin{abstract}
  In \cite{Rapcsak_1961} A.~Rapcs\'ak obtained necessary and sufficient
  conditions for the projective Finsler metrizability in terms of a second
  order partial differential equations.  In this paper we investigate the
  integrability of the Rapcsák system, consisting of the Rapcsák equations and
  the homogeneity condition, by using the Spencer version of the
  Cartan-K\"ahler theorem. We also consider the extended Rapcsák system, where
  the first integrability conditions are included.
\end{abstract}

\bigskip

\begin{description}
\item[2010 MSC class:] 49N45, 58E30, 53C60, 53C22

\item[Keywords:] sprays, projective metrizability, partial differential
  operator, formal integrability.
\end{description}


\section{Introduction}
\label{sec:introduction}

Last year we celebrated the 100\ts{th} years of the birth of András Rapcsák.
He was one of the founder of the Finsler Geometry research school in Debrecen.
Rapcsák's results are still relevant and up to date, as several recent
citations show it.  His most important results concern the projective Finsler
metrizability problem, where one seeks for a Finsler metric whose geodesics
are projectively equivalent to the solutions of a given second order
homogeneous ordinary differential equations (SODE).

The projective Finsler metrizability problem can be considered as a particular
case of the inverse problem of the calculus of variations.  Rapcs\'ak
\cite{Rapcsak_1961} obtained necessary and sufficient conditions for the
projective Finsler metrizability in terms of a second order PDE system, called
now \emph{Rapcsák equations} \cite{CMS_2012, Szilasi_2002, Shen}. The
coordinate free formulations of these equations can be found in \cite{Kl_1968,
  Szilasi_2002}.  Rapcsák's approach is simple and natural: one finds
conditions directly on the Finsler function that one seeks for.  Recently
several new results appeared about the projective Finsler metrizability
problem \cite{BMz_2011, crampin08, CMS_2012, CMS_2013, Matveev_2012, Mestdag}.
We remark that, altough in these papers the Rapcsák equation is usually
mentionned, the new results were obtained by different approaches, for example
by the so-called multiplier method, where one seeks for the existence of a
variational multiplier matrix. In \cite{CMS_2012} the generalized Helmholtz
system was considered and in \cite{BMz_2011} a system in terms of a semi-basic
1-form was investigated.  In this paper, in the perspective of the projective
metrizability problem, we consider the \emph{Rapcsák system}, which consists
of the homogeneity equation \eqref{eq:k-homog} and a second order differential
equation \eqref{eq:0}, called the \emph{Rapcsák equation}.  We investigate the
integrability of the Rapcsák system by using the Spencer version of the
Cartan-K\"ahler theorem.

The structure of the paper is as follows. In Section \ref{sec:prelim} we give
a brief introduction to the Fr\"olicher-Nijenhuis theory and to the canonical
structures on the tangent bundle of a manifold.  We also introduce the main
structures one needs to discuss the geometry of a spray: connection, Jacobi
endomorphism, curvature.  We also recall the basic tool of the Cartan-K\"ahler
theory.

In Section \ref{sec:rapcsak} we use the geometric setting presented in Section
\ref{sec:prelim} to show that the Rapcsák system gives necessary and
sufficient condition for the projective metrizability problem. Alternative
proves can be found in \cite{SzLK_2014, Szilasi_2002}.  We discuss the
integrability of the Rapcsák system by using conditions provided by
Cartan-K\"ahler theorem.  We conclude the chapter by showing that there is
only one obstruction to the formal integrability of $P_1$. This obstruction is
expressed in terms of the nonlinear connection induced by the spray.

In Section \ref{sec:ext_rapcsak} we investigate the formal integrability of
the \emph{extended Rapcsák system} composed by the Rapcsák system and its
integrability conditions found in Section \ref{sec:rapcsak}. We show that the
obstruction to the integrability of this new system can be expressed in terms
of the curvature tensor of the nonlinear connection induced by the spray.  For
some classes of sprays the curvature obstruction is identically satisfied:
flat sprays, isotropic sprays, and sprays on 1- and 2-dimensional manifolds.
For each of these classes the Weyl curvature of the spray is zero. Although,
for some of these classes the projective metrizability problem has been
discussed before by some authors, our approach in this work is different. This
approach gives also the possibility to push forward the computation and
consider sprays with non vanishing Weyl curvature.  We remark however that in
that case the computations are long and complex because the symbol of the
correspondent differential operator may be not involutive. By computing the
appropriate Spencer cohomology groups one can prove that it is not $2$-acyclic
either \cite{MiMz_preprint} which shows that higher order compatibility conditions
arises. The analysis of this new system will be the subject of an other
publication.

\section{Preliminaries}
\label{sec:prelim}

Throughout this paper $M$ will denote an $n$-dimensional smooth
manifold. $C^\infty(M)$ denotes the ring of real-valued smooth functions,
$\mathfrak X(M)$ is the $C^\infty(M)$-module of vector fields on $M$, $\pi:TM
\to M$ is the tangent bundle of $M$, $\TM=TM \! \setminus\! \{0\}$ is the slit
tangent space.  We will essentially work on the manifold $TM$ and on its
tangent space $TTM$.  When there is no danger of confusion, $TTM$ and $T^*TM$
will simply be denoted by $T$ and $T^*$, respectively.  $VTM=\Ker \pi_*$ is
the vertical sub-bundle of $T$. We denote by $\Lambda^k(M)$, $S^k(M)$ and
$\Psi^k(M)$ the $C^\infty (M)$-modules of the skew-symmetric, symmetric and
vector valued $k$-forms respectively.  Similarly, we denote by
$\Lambda^k_v(TM)$ and $\Psi^k_v(TM)$ the $C^\infty (TM)$-modules of semi-basic
$k$-forms and semi-basic vector valued $k$-forms.

The Fr\"olicher-Nijenhuis theory provides a complete description of the
derivations of $\Lambda(M)$ with the help of vector-valued differential forms,
for details we refer to \cite{FN}.  The $i_*$ and the $d_*$ type derivation
associated to a vector valued $l$-form $L$ will be denoted by $i_L$ and
$d_L$. They can be introduce in the following way: if $L\in\Psi^{l}(M)$, then
\begin{align*} 
  i_L\omega(X_1,\dots,X_l)=\omega(L(X_1,\dots,X_l)),
\end{align*} 
where $X_1,\dots,X_l\in\mathfrak{X}(M)$, $\omega \in \Lambda^1(M)$.
Furthermore, $d_L$ is the commutator of the derivations $i_L$ and $d$, that is
\begin{align*} 
  d_L:=[i_L,d]=i_Ld-(-1)^{l-1}di_L.
\end{align*} 
We remark, that for $X\!\in\!\mathfrak{X}(M)$ we have $d_X\!=\!\mathcal{L}_X$
the Lie derivative, and $i_x$ is the substitution operator. The
\textit{Fr\"olicher--Nijenhuis bracket} of $K\in \Psi^{k}(M)$ and $L\in
\Psi^{l}(M)$ is the unique $[K,L]\in \Psi^{k+l}$ form, such that
\begin{align*}
	[d_K,d_L]=d_{[K,L]}.
\end{align*}
In the special case, when $K\in \Psi^{1}(M)$, $X,Y\in\mathfrak{X}(M)$ we have
$[K,X]\in\Psi^{1}(M)$ defined as
\begin{align*}
  [K,X](Y)=[KY,X]-K[Y,X].
\end{align*}

\subsubsection*{Spray and associated geometric quantities}

Let $J\colon TTM\rightarrow TTM$ be the \emph{vertical endomorphism} and $C
\in \mathfrak X(TM)$ the \textit{Liouville vector field}.  In an induced local
coordinate system $(x^i,y^i)$ on $TM$ we have
\begin{math}
  J = dx^i \otimes \frac {\partial}{\partial y^i },
\end{math}
and 
\begin{math}
  C = y^i \frac {\partial}{\partial y^i }.
\end{math}
Euler's theorem for homogeneous functions implies that $L\!\in\!C^\infty(TM)$
is a $1$-ho\-mo\-geneous function in the $y$ variable if and only if
\begin{equation}
  \label{eq:k-homog}
  y^i \frac{\partial L}{\partial y^i}-L = 0.
\end{equation}
The vertical endomorphism satisfies the following properties: $J^{2} = 0$,
$\Ker J \!=\! \Im J\!=\!VTM$ and $[J,C]\!=\!J$.

A \textit{spray} is a vector field $S$ on $\TM$ satisfying the relations $JS =
C$ and $[C,S]=S$. The coordinate representation of a spray $S$ takes the form
\begin{equation*}
  \label{eq:S}
  S = y ^i \frac {\partial}{\partial x^i } +f^i (x,y)
  \frac {\partial}{\partial y^i },
\end{equation*}
where the functions $f^i(x,y)$ are homogeneous of degree 2 in $y$.  The
\textit{geodesics of a spray} are curves $\gamma : I \to M$ such that $S \circ
\dot \gamma = \ddot \gamma$. Locally, they are the solutions of the equations
\begin{equation}
  \label{eq:sode}
  \hphantom{\qquad i=1,...,n} \ddot x^{i} =f^i\left(x, \dot x\right), 
  \qquad  i=1,\dots,n.
\end{equation}
Two sprays $S$ and $\widetilde{S}$ are called \textit{projective equivalent},
if their geodesics coincide up to an orientation preserving
reparametrization. It is not difficult to show that $S$ and $\widetilde{S}$
are projective equivalent if and only if they are related, by the formula:
\begin{equation}
  \label{eq:equiv_spray}
  \widetilde S = S -2\P C ,
\end{equation}
where $\P \in C^{\infty}(\TM)$ is a 1-homogeneous function.

To every spray $S$ a \textit{connection} $\Gamma: = [J,S]$ can be
associated. We have $\Gamma^2 = \mathrm{Id}$. The eigenspace of $\Gamma$
corresponding to the eigenvalue $-1$ is the \emph{vertical space} $VTM$, and
the eigenspace corresponding to +1 is called the \emph{horizontal space}. For
any $x\in TM$ we have
\begin{math}
  T_xTM = H_xTM \oplus V_{x}TM.
\end{math}
The corresponding horizontal and vertical projectors associated to $\Gamma$
are denoted by $h$ and $v$. One has
\begin{equation}
  \label{eq:h_v}
  h= \frac{1}{2}(\textrm{Id} + \Gamma ), \qquad 
  v= \frac{1}{2}(\textrm{Id} - \Gamma ).
\end{equation}
The \textit{curvature} $R = \frac{1}{2}[h,h]$ of the connection is the
Nijenhuis torsion of the horizontal projection $h$. The \textit{Jacobi
  endomorphism} (or Riemann curvature \cite{Shen}) is defined as
$\Phi=i_{S}R$.  The Jacobi endomorphism determines the curvature by the
formula
\begin{math}
  R=\frac{1}{3}[J,\Phi].
\end{math}
The spray $S$ is called \emph{flat} if its Jacobi endomorphism has the form
$\Phi=\lambda J$ and \emph{isotropic}, if $\Phi=\lambda J-\alpha\otimes C$
with some $\lambda \in C^{\infty}(\TM)$, $\alpha\in \Lambda^1_v(\TM)$.

\subsubsection*{Finsler structure}

A \emph{Finsler function} on a manifold $M$ is a continuous function $F\colon
TM \to \R$, which is smooth and positive away from the zero section,
homogeneous of degree 1, and strictly convex on each tangent space. The energy
function $E:TM \to \R$ associated to a Finsler structure $F$ is defined as
$E:=\frac{1}{2}F^2$. The
\begin{displaymath}
  g_{ij} := \frac{\partial^2 E}{\partial y^i \partial y^j}
\end{displaymath}
is positive definite at any points $(x,y)\in \TM$. The pair $(M,F)$ is called
Finsler manifold. The geodesics on the Finsler manifold $(M,F)$ are the
solutions of the Euler-Lagrange equation
\begin{equation}
  \label{eq:E-L_loc}
  \hphantom{\qquad i=1,...,n}
  \frac{d}{dt}\frac{\partial E}{\partial \dot{x}^i }
  - \frac{\partial E}{\partial x^i} = 0, \qquad i=1,\dots,n.
\end{equation}
It is not difficult to see that for any function $E\in C^\infty(\TM)$ the
1-form
\begin{equation}
  \label{eq:E_L}
  \omega_E = i_S dd_JE +d{\mathcal L}_C E -dE,
\end{equation}
is semi-basic, and its coordinate representation takes the form $\omega_E =
\omega_i \, dx^i$ where the coefficients $\omega_i$ are the functions
appearing in the left-hand side of the Euler-Lagrange equation
\eqref{eq:E-L_loc}.  Therefore $S$ corresponds to the geodesic equation of $E$
if and only if the equation
\begin{equation}
  \label{eq:omega_E=0}
  \omega_E = 0
\end{equation}
holds.  The spray $S$ is called \emph{Finsler metrizable}, if there exists a
Finsler function such that for the corresponding energy function
\eqref{eq:omega_E=0} holds, and $S$ is \emph{projective Finsler metrizable},
if it is projective equivalent to a \emph{Finsler metrizable} spray.

\subsubsection*{Formal integrability}

To investigate the integrability of the Rapcsák system we shall use Spencer's
technique of formal integrability in the form explained in \cite{GrMz_2000}.
For a detailed account see \cite{BCGGG}.  We recall here the basic notions in
order to fix the terminology.

Let $B$ be a vector bundle over $M$.  If $s$ is a section of $B$, then
$j_k(s)_x$ will denote the $k$\ts{th} order jet of $s$ at the point $x \in M$.
The bundle of $k$\ts{th} order jets of the sections of $B$ is denoted by
$J_kB$.  In particular $J_k(\R_M)$ will denote the $k$\ts{th} order jets of
real valued functions, that is the sections of the trivial line bundle.  Let
$B_1$ and $B_2$ be vector bundles over $M$ and
\begin{math}
  P\colon Sec\,(B_1) \to Sec \,(B_2)
\end{math}
a differential operator. An $s\in Sec (B_1)$ is a \emph{solution} to $P$ if
$Ps\equiv 0$.

If $P$ is a linear differential operator of order $k$, then a morphism
\begin{math}
  p_k(P) \colon J_{k}(B_1) \to B_2
\end{math}
can be associated to $P$.  The $l$\ts{th} order prolongation $p_{k+l}(P)\colon
J_{k+l}(B_1) \to J_l(B_2)$ can be introduced in a natural way by taking the
$l$\ts{th} order derivatives.  $Sol_{k+l, x}(P):=\mathrm{Ker}\, p_{k+l,x}(P)$
denotes the set of \emph{formal solutions of order $l$ at $x \in M$}. Obviously,
we have
\begin{displaymath}
  Ps \equiv 0 \qquad \Rightarrow \qquad j_{l,x}(s) \in Sol_{l, x}(P),
\end{displaymath}
for every $l\geq k$ and $x\in M$.  The differential operator $P$ is called
\textit{formally integrable} if $Sol_l(P)$ is a vector bundle for all $l \geq
k$, and the restriction $\overline{\pi}_{l,x}\colon Sol_{l+1,x}(P)\rightarrow
Sol_{l,x}(P)$ of the natural projection is onto for every $l\geq k$.  In that
case any $k$\ts{th} order solution or \emph{initial data} can be lifted into
an infinite order solution.  In the analytic case, formal integrability
implies the existence of solutions for arbitrary initial data
(see.~\cite{BCGGG}, p.~397).  To prove the formal integrability, one can use
the Cartan-K\"ahler theorem. To present it, we have to introduce some
notations.

Let $\sigma_{k}(P)$ denote the symbol of $P$ determined by the highest order
terms of the operator. It can be interpreted as a map $\sigma_{k}(P)\colon
S^{k}T^*\!M \otimes B_1 \to B_2$.  $\sigma_{k+l}(P):S^{k+l}T^*\!M \otimes B_1
\to S^lT^*\!M \otimes B_2$ denotes the symbol of the $l$\ts{th} order
prolongation of $P$. If $\mathcal E \!=\! \{e_1 \dots e_n\}$ is a basis of
$T_x M$, we set 
\begin{alignat*}{2}
  &g_{k,x} (P) & &= \Ker \sigma_{k,p}(P),
  \\
  &g_{k,x} (P)_{e_1\dots e_j} && = \bigl\{A\in g_{k,p}(P) \mid i_{e_1}A =
  \dots = i_{e_j}A = 0 \bigl\}, \quad j=1,\dots,n,
\end{alignat*}
The basis $\mathcal E$ is called \emph{quasi-regular} if one has
\begin{displaymath}
  \textrm{dim} \, g_{k+1,x}(P) = \textrm{dim} \, g_{k,x}(P) + \sum
  _{j=1}^{n} \textrm{dim} \, g_{k,x}(P)_{e_1\dots e_j} .
\end{displaymath}
A symbol is called \textit{involutive}\footnote{In the works of Cartan, and
  more generally in the theory of exterior differential systems,
  "involutivity" means more than the existence of a quasi-regular basis and it
  refers to "integrability" (cf.  \cite{BCGGG}, p.107, 140). Here we are
  following the terminology of Goldschmidt (cf.  \cite{BCGGG}, p.\,409).}  if
there exists at any $x\in M$ a quasi-regular basis.  The notion of
involutivity allows us to check the formal integrability in a simple way by
using the following
\begin{theorem}
  \label{thm:cartan_kahler}
  [Cartan-K\"ahler].  Let $P$ be a $k$\ts{th} order linear partial
  differential operator.  Suppose that $P$ is regular, that is $Sol_{k+1}(P)$
  is a vector bundle over $Sol_k(P)$.  If the map $\overline {\pi }_{k} \colon
  Sol_{k+1}(P)\rightarrow Sol_k(P)$ is surjective and the symbol is
  involutive, then $P$ is formally integrable.
\end{theorem}
It can be shown that the condition of the existence of a quasi-regular basis
can be replaced by a weaker condition. The obstructions to the higher order
successive lift of the $k$\ts{th} order solution are contained in some of the
cohomological groups of a certain complex called \emph{Spencer complex}.

\vspace{0.2cm}

Using a classical result in homological algebra gives, the surjectivity of
$\overline{\pi}_{k+1}$ can be verified in the following way \cite{GrMz_2000}:
\begin{proposition}
  \label{sec:phi}
  There exits a morphism $\varphi\colon Sol_k(P)\rightarrow \mathrm{Coker}\,
  (\sigma_{k+1}(P))$, such that the sequence
  \begin{displaymath}
    Sol_{k+1}(P)\xrightarrow{\ \overline{\pi}_k \ } Sol_k(P)
    \xrightarrow{\ \varphi \ } \mathrm{Coker}\, (\sigma_{k+1}(P))
  \end{displaymath}
  is exact.  Therefore $\overline{\pi}_k$ is surjective if and only if
  $\varphi\equiv 0$.
\end{proposition}

\begin{remark}
  The map $\varphi$ is called \emph{obstruction map} and $\mathrm{Coker}\,
  (\sigma_{k+1}(P))$ is called \emph{obstruction space}, because a $k$\ts{th}
  order solution $s\!\in\!  Sol_k(P)$ can be prolongated into a
  $(k\!+\!1)$\ts{st} order solution iff $\varphi(s)\!=\!0$. In particular, if
  $\mathrm{Coker}\, (\sigma_{k+1}(P))\!=\!\{0\}$ then there is no obstruction
  to the prolongation.
\end{remark}
In the practice the map $\varphi$ and therefore the integrability conditions
can be computed as follows:
\begin{remark}
  \label{sec:phi_compute}
  Let be $\tau\colon T^{*}\!\otimes \! B_2\rightarrow K$ a morphism such that
  $\Ker \tau \!=\!\Im \sigma_{k+1}(P)$. Then $K$ is isomorphic to
  $\mathrm{Coker}\, (\sigma_{k+1}(P))$. Moreover, if $s_{k,x}\!=\!j_k(s)_x$ is
  a $k$\ts{th} order solution, that is $(Ps)_x=0$, then
  \begin{align*}
    \varphi(s_{k,x})=\tau(\nabla (Ps))_x,
  \end{align*}
  where $\nabla$ is an arbitrary linear connection on the bundle $B_2$.
\end{remark}

\medskip

\begin{remark}
  Let $(x^i)$ be a local coordinate system on $M$, $(x^i,y^i)$ the associated
  coordinate system on $TM$ in the neighborhood of $v \in TM$.  If $j_k(F)_v
  \in J_{k}(\R_{TM})$ is a $k$th order jet of a real valued function $F$ on
  $TM$ we set
  \begin{equation}
    \label{eq:J_k_loc}
    s_{i_1\cdots i_a \underline{i_{a+1} \cdots i_l}}:= \frac{\partial^l F}{\partial
      x^{i_1} ... \,  \partial x^{i_a} \partial y^{i_{a+1}} ... \,   \partial
      y^{i_l}}(v), \qquad 1 \leq l \leq k. 
  \end{equation}
  Then
  \begin{math}
    (s,s_i,s_{\u i})
  \end{math}
  and 
  \begin{math}
    (s,s_i,s_{\u i}, s_{ij}, s_{\u i j}, s_{\u{ij}})
  \end{math}
  give coordinate systems on $J_{1}(\R_{TM})$ and $J_{2}(\R_{TM})$
  respectively.
\end{remark}

\bigskip

\section{Differential operator of the projective metrizability}
\label{sec:rapcsak}

In this section we derive the PDE system describing the necessary and
sufficient condition for a spray to be projective Finsler metrizable. We have
the following
\begin{proposition}
  \label{thm:proj_met}
  A spray $S$ is projective Finsler metrizable if and only if there exists a
  1-homogeneous Lagrange function $\widetilde{F}\colon TM\!\to\! \R$, such that
  $\frac{\partial^2\widetilde{F}^2}{\partial y^i \partial y^j}$ is positive
  definite on $\TM$ and
  \begin{equation}
    \label{eq:0} 
    i_{S}dd_{J}\widetilde{F}=0.
  \end{equation}
\end{proposition}

\begin{proof}
  The spray $S$ is projective Finsler metrizable if and only if there exists a
  Finsler metrizable spray $\widetilde{S}$ which is projective equivalent to
  $S$.  Because of the projective equivalence, there exists a function
  $\mathcal P$, such that $\widetilde{S}=S-2\mathcal P C$.  Let us denote by
  $\widetilde{F}$ the Finsler function associated to $\widetilde{S}$. It is well
  known that $\widetilde{F}$ is invariant by the parallel translation
  associated to the connection $\widetilde{\Gamma}=[J,\widetilde{S}]$ and
  therefore we have $d_{\widetilde{h}}\widetilde{F}=0$. Using the relation
  \begin{displaymath}
    \widetilde{h}=h-\P  J-d_J\P \otimes C
  \end{displaymath}
  between the horizontal projectors (\cite{BM_2012}, chapter 4) and the
  1-homogeneity of $\widetilde{F}$, we get
  \begin{alignat}{1}
    \label{eq:lambda}
    0 = d_{\widetilde{h}}\widetilde{F} = d_{h}\widetilde{F}-d_{\P
      J}\widetilde{F}-d_J\P C\widetilde{F}
    =d_{h}\widetilde{F}-\P d_{ J}\widetilde{F}-\widetilde{F}d_J\P
    = d_{h}\widetilde{F}- d_{J}(\P \widetilde{F}).
  \end{alignat}
  Substituting $S$ in \eqref{eq:lambda} and using $JS=C$ and the homogeneity
  of $\widetilde{F}$ and $\P$, we get
  \begin{displaymath}
    i_Sd_{\widetilde{h}}\widetilde{F} =S \widetilde{F}
    - C(\P \widetilde{F}) =S \widetilde{F}-2\P \widetilde{F} =0
  \end{displaymath}
  and we can find, that the projective factor is
  \begin{math}
    \P =\frac{1}{2\widetilde{F}}S \widetilde{F}.
  \end{math}
  Replacing $\P$ in \eqref{eq:lambda} by the above expression we
  get
  \begin{alignat*}{1}
    d_h\widetilde{F}-d_J\Big(\frac{1}{2\widetilde{F}} (\widetilde{F}\,
    d_S\widetilde{F})\Big)
    = d_h\widetilde{F}-\frac{1}{2}d_J (d_S\widetilde{F})=0.
  \end{alignat*}
  Using \eqref{eq:h_v} and the relation $d_{[J,S]}= d_Jd_S-d_Sd_J$ we can
  obtain
  \begin{displaymath}
    \begin{aligned}[b]
      0 & =d_{\Gamma+I}\widetilde{F}-d_J d_S\widetilde{F}
      =d_{[J,S]}\widetilde{F}+d\widetilde{F}-d_J d_S\widetilde{F}
      \\
      & = -(i_Sd+di_S)d_J\widetilde{F}+d\widetilde{F}
      =-i_Sdd_J\widetilde{F}-dC\widetilde{F}+d\widetilde{F}
      =-i_Sdd_J\widetilde{F}.  
    \end{aligned}
    \qedhere
  \end{displaymath}
\end{proof}
We note, that a coordinate version of the above theorem was proved by
A. Rapcs\'ak in \cite{Rapcsak_1961} and a coordinate free version was given in
\cite{Kl_1968, SzLK_2014,Szilasi_2002}. Here we presented a different proof.

\medskip

\begin{definition}
  Let $S$ be a spray on $M$.  The partial differential system composed by the
  equation \eqref{eq:0} and the 1-homogeneity condition \eqref{eq:k-homog} is
  called the \emph{Rapcsák system}.
\end{definition}
According to Proposition \ref{thm:proj_met} the projective metrizability leads
to the investigation of the Rapcsák system. 
\begin{remark}
  The Rapcsák system is equivalent to the system composed by the
  Euler-Lagrange equations \eqref{eq:E-L_loc} and the 1-homogeneity condition
  \eqref{eq:k-homog}.
\end{remark}

We remark that the system composed by the Euler-Lagrange equations and the
$k$-homogeneity condition for $k \!\neq\! 1$ can be reduced to a first order
partial differential system which can be interpreted in terms of the holonomy
distribution associated to the spray $S$. When $k=2$, (this case corresponds
to the Finsler metrizability problem) the computation can be found in
\cite{Mz08}.  The same reasoning can be applied for other value of $k$,
$k\!\neq\! 1$. But this method cannot be used for the value
$k=1$. Nevertheless, in some special situations, the Rapcsák system can also
be reduced to a first order PDE system.  This is the case for example for the
canonical spray of a Lie group, if one seeks for an invariant solution to the
projective Finsler metrizability problem. In that case, the Rapcsák system can
be reduced to a first order system, and one can show, that the invariant
Riemann, Finsler and projective Finsler meterizability problems are equivalent
\cite{MiMz}.

\subsection*{Integrability conditions of the Rapcsák system}

\noindent
Let us consider the differential operator $P_1$ corresponding to the Rapcsák
system:
\begin{equation}
  P_1=(P_{S},P_{C}),
\end{equation}
where \vspace{-12pt}
\begin{alignat}{2}
  & P_S\colon C^{\infty}(TM) \longrightarrow Sec \, T^{*}, & \qquad
  P_S(F)&=i_Sdd_JF,
  \\
  & P_C\colon C^{\infty}(TM) \longrightarrow C^{\infty}(TM), & \qquad
  P_{C}(F)&={\mathcal L}_CF-F.
\end{alignat}
From the local expression it is clear that $P_C$ is a first and $P_S$ is a
second order differential operator.  The associated morphisms are defined on
the first and second order jet spaces respectively. Using the coordinate
system \eqref{eq:J_k_loc} we get
\begin{alignat*}{2}
  p_1(P_C) &\colon J_1 (\R_{TM}) \longrightarrow \R,& \qquad j_1(F)&
  \longrightarrow y^i F_{\u i} -F,
  \\
  p_2(P_S) & \colon J_2 (\R_{TM}) \longrightarrow T^*,& \qquad j_2(F)&
  \longrightarrow (y^i F_{i \u j}\!+\!f^i F_{\u ij} \!-\! F_j)dx^i \!-\!
  (F_{\u i} \!+\! y^j F_{\u ij} \!-\!F_i)dy^i.
\end{alignat*}
The interesting feature of the Rapcsák system is that it is composed by
differential operators of different orders. To find the integrability
conditions of the system we consider the prolongation of the lower order
equation. The morphism associated to this system is
\begin{displaymath}
  p_2(P_1) =   p_2(P_S) \times p_2(P_C) \colon \quad J_2(\R_{TM})
  \longrightarrow T^* \! \times \! J_1(\R_{TM}).
\end{displaymath}
\begin{proposition}
  \label{sec:first-integr-cond}
  A 2\ts{nd} order solution of $P_1\!=\!(P_S, P_C)$ at $x\!\in\! \TM$ can be
  lifted into a 3\ts{rd} order solution, if and only if one has
  $i_{\Gamma}dd_JF=0$ at $x$, where $\Gamma=[J,S]$ is the canonical nonlinear
  connection associated to $S$.
\end{proposition}
\noindent
\emph{Proof.} The symbols are defined by the highest order part of the
operators. For $P_C$ we find
\begin{alignat*}{2}
  \sigma_1(P_C)\colon & T^* \longrightarrow \R, \quad & & \sigma_1( P_C)A_1 =
  A_1(C).
  \intertext{The symbol of $P_S$ and the prolongation of the symbol of $P_C$
    are}
  \sigma_2(P_C)\colon & S^2T^* \longrightarrow T^*, \qquad &
  &\big(\sigma_2(P_C)A_2\big)(X) = A_2(X, C),
  \\
  \sigma_2(P_S)\colon & S^2T^* \longrightarrow T^*, \qquad &
  &\big(\sigma_2(P_S)A_2\big)(X) = A_2(S, JX)-A_2(X, C),
  \intertext{for every $X\!\in\! T$, $A_1\!\in\! T^{*}$, $A_2\!\in\! S^{2}T^{*}$.  The
    prolongations of the symbols at third order level are}
  \sigma_3(P_C): & S^3T^* \longrightarrow T^* \!\otimes S^{2}T^*, \quad &&
  \big(\sigma_3(P_C)A_3\big)(X,Y) = A_3(X,Y, C),
  \\
  \sigma_3(P_S): & S^3T^* \longrightarrow T^*\! \otimes T^*, \quad &&
  \big(\sigma_3(P_S)A_3\big)(X,Y) = A_3(X, S, JY)-A_3(X, Y,C),
\end{alignat*}
where $X,Y\! \in\! T$, $A_3\!\in\! S^{3}T^{*}$ and we have 
\begin{align*}
  \sigma_3(P_1)=\big(\sigma_3(P_S), \sigma_3(P_C)\big) \colon
  S^3T^*\longrightarrow (T^{*} \!\otimes T^{*}) \times S^{2} T^{*}.
\end{align*}
Let us consider the map
\begin{math}
  \tau_1 := (\tau^1_{S}, \, \tau^2_{S}, \, \tau^{1}_{SC}, \, \tau^{2}_{SC})
\end{math}
where
\begin{alignat}{2}
  \label{eq:symb_1_6}
  &\tau^{1}_{S} (B_S,B_C)(X,Y)&&= B_S(JX, hY) \!-\! B_S(hY, JX) \!-\!B_S(JY,
  hX) \!+\! B_S(hX, JY),
  \\
  \label{eq:symb_1_7}
  &\tau^{2}_{S} (B_S,B_C)(X)&&=B_S(X,S),
  \\
  \label{eq:symb_1_9}
  &\tau^{1}_{SC} (B_S,B_C)(X,Y)&&=B_{S}(X,JY)+B_{C}(X,JY),
  \\
  &\tau^{2}_{SC} (B_S,B_C)(X,Y)&&=B_{S}(C,hX)-B_{C}(S,JX)+B_C(hX,C),
\end{alignat} 
for $B_S\!\in\! T^{*}\!\otimes T^{*}$\!, $B_C\!\in\! S^2T^{*}$\!, $X,Y\!\in\!
T$.  
\begin{lemma}
  We have $\Im \sigma_3(P_1)\!=\!\Ker \tau_1$ that is, if we denote
  $K_1\!=\!\mathrm{Im}\, \tau_1$ then the sequence
  \begin{align}
    \label{ex:1}
    S^{3}T^{*}\xrightarrow{ \ \sigma_3(P_1) \ } (T^{*} \!\otimes T^{*})\!\times
    S^{2} T^{*}\xrightarrow{ \ \tau_1 \ } K_1 \rightarrow 0
  \end{align}
  is exact.
\end{lemma}

\begin{proof}
  By the construction, we have to check the exactness in the second term. It
  is easy to compute that $\sigma_3(P_1) \circ \tau_1 =0$ and therefore
  \begin{math}
    \Im \sigma_3(P_1)\subset \Ker \tau_1.
  \end{math}
  Let us compute $\dim \mathrm{Ker} \sigma_3(P_1)$.  We consider the basis
  \begin{equation}
    \label{eq:basis_1}
    \mathcal B:=\left\{h_1,\dots ,h_n,v_1,\dots,v_n\right\} \quad \subset T_x,
  \end{equation}
  where $h_i$ are horizontal, $h_n\!=\!S$, $Jh_i\!=\!v_i$, $i\!=\!1,\dots,n$
  (and therefore $v_n=C$).  In the sequel we denote the components of a
  symmetric tensor $A\in S^k T^*$ with respect to \eqref{eq:basis_1} as
  \begin{equation}
    \label{eq:B_notation}
    A_{i_1\dots i_j \u{i_{j+1}}\dots \u{i_k}}:=
    A(h_{i_1}, \dots, h_{i_j},v_{i_{j+1}},\dots, v_{i_k}).
  \end{equation}
  It is clear that
  \begin{math}
    \Ker \sigma_3(P_1)\!=\!\Ker \sigma_3(P_S)\cap\Ker \sigma_3(P_C).
  \end{math}
  The symmetric tensor $A \in S^3T^*$ is in $\Ker \sigma_3(P_C)$ if
  \begin{equation}
    \label{eq:symb_1}
    A_{ij\u n}= A_{i\u j \u n}= A_{\u i\u j \u n}=0,
  \end{equation}
  and $A \in S^3T^*$ is an element of $\Ker \sigma_3(P_S)$ if
  \begin{alignat}{1}
    \label{eq:symb_1_1}
    \sigma_3(P_S)(A) (h_i, h_j)& = A(h_i, h_n, v_j)- A(h_i, h_j, v_n)= A_{in
      \u j}-A_{i j \u n}=0,
    \\
    \label{eq:symb_1_2}
    \sigma_3(P_S)(A) (h_i, v_j)& = - A(h_i, v_j, v_n)=- A_{i \u j \u n}=0,
    \\
    \label{eq:symb_1_3}
    \sigma_3(P_S)(A) (v_i, h_j)& = A(v_i, h_n, v_j)- A(v_i, h_j, v_n)=A_{\u i
      n \u j}-A_{\u i j \u n}=0,
    \\
    \label{eq:symb_1_4}
    \sigma_3(P_S)(A) (v_i, v_j)& =- A(v_i, v_j, v_n)=- A_{\u i \u j \u n}=0,
  \end{alignat}
  for $i,j\!=\!1,\dots,n$. Taking into account the symmetry of $A$ we have
  $2\frac{n(n+1)}{2}+n^{2}$ independent equations in \eqref{eq:symb_1}.
  Moreover, counting the independent equations in
  \eqref{eq:symb_1_1}-\eqref{eq:symb_1_4} we get that (\ref{eq:symb_1_2}) and
  (\ref{eq:symb_1_4}) trivially hold because of (\ref{eq:symb_1}). From
  (\ref{eq:symb_1_1}) we have only $n^2-n$ independent equations because for
  $j\!=\!n$ the equations are trivially satisfied, and from
  (\ref{eq:symb_1_3}) we have $\frac{n(n-1)}{2}$ independent equations because
  again, for $j\!=\!n$ they are trivially satisfied.  Consequently, we have
  \begin{math}
    2\frac{n(n+1)}{2}+2n^2-n +\frac{n(n-1)}{2}=\frac{7n^{2}-n}{2}
  \end{math}
  independent equations in the system
  \eqref{eq:symb_1}-\eqref{eq:symb_1_4}. Therefore we get
  \begin{equation}
    \label{eq:6}
	\dim (g_3(P_1))=\dim \Ker \sigma_3(P_1)=\dim
    S^{3}T^{*}-\frac{7n^{2}-n}{2}=\frac{8n^{3}-9n^{2}+7n}{6}
  \end{equation}
  and 
  \begin{equation}
    \label{eq:1}
	\mathrm{rank}\sigma_3(P_1)=\frac{7n^{2}-n}{2}.
  \end{equation}
  On the other hand, let us compute $\dim \Ker \tau_1$.  The pivot terms for
  the equation $\tau^{1}_{S}\!=\!0$ are $B_S(v_i,h_j)$,
  $i\!<\!j\!<\!n$. Furthermore, $B_S(v_i,h_n)$, $B_S(h_i,h_n)$,
  $i\!=\!1,\dots,n$ are pivot terms for $\tau^{2}_{S}\!=\!0$. Therefore the
  number of independent equations for $\Ker \tau^{1}_{S}$ and $\Ker
  \tau^{2}_{S}$ are $\frac{(n-1)(n-2)}{2}$ and $2n$, respectively.  Moreover,
  the pivot terms for the equations $\tau^{1}_{SC}\!=\!0$
  and $\tau^{2}_{SC}\!=\!0$ are $B_S(h_i,v_j)$, $B_S(v_i,v_j)$,
  $i,j\!=\!1,\dots,n$, and $B_S(v_n,h_i)$, $i\!=\!1,\dots,n\!-\!1$, giving in
  addition $2n^2\!+\!n\!-\!1$ independent equations.
  \begin{equation}
    \label{eq:2}
    \dim \Ker \tau_1\!=\!\dim S^2T^*\!+\!\dim 
    (T^{*}\!\otimes\! T^{*})-\left[\frac{(n\!-\!1)(n\!-\!2)}{2}
      \!+\!2n^{2}\!+\! 3n\!-\!1\right]
    =\frac{7n^{2}\!-\!n}{2}.
  \end{equation}
  Comparing \eqref{eq:1} and \eqref{eq:2} we get $\Im \sigma_3(P_1)\!=\!  \Ker
  \tau_1$.
\end{proof}
\bigskip

\noindent
\emph{Proof of Proposition \ref{sec:first-integr-cond}}.  The morphisms, the
symbols and the obstruction map associated to the Rapcsák system can be
represented in the following commutative diagram:
\begin{displaymath}
  \begin{CD}
    g_3(P_1)@>>> S^{3}T^{*} @>\sigma_3(P_1)>> (T^{*}\!\otimes T^{*}\!)\times
    S^2 T^{*} @>\tau_1>> K_1 \longrightarrow 0
    \\
    @VVV @VV \epsilon V @VV\epsilon V
    \\
    Sol_3(P_1) @> i >> J_3(\RTM) @>p_3(P_1)>> J_1(T^{*})\times J_2(\RTM)
    \\
    @VV \overline{\pi}_2V @VV \pi_2 V @VV \pi_0\!\times\pi_1 V
    \\
    Sol_2(P_1) @> i >> J_2(\RTM) @>p_2(P_1)>> T^{*}\times J_1(\RTM)
  \end{CD}
\end{displaymath}

\bigskip

\noindent
Let $s\!=\!j_{2}(F)_x\!\in \!Sol_{2,x}(P_1)$ be a second order solution of
$P_1$ at $x$, that is
\begin{equation}
  \label{eq:P_1_x}
  (i_Sdd_JF)_x\!=\!0, \qquad (\L_CF\!-\!F)_x\!=\!0, \qquad 
  \big(\nabla(\L_CF\!-\!F)\big)_x\!=\!0.
\end{equation}
The integrability condition can be computed in terms of $ \tau_1 \!=\!
(\tau^1_{S}, \, \tau^2_{S}, \, \tau^{1}_{SC}, \, \tau^{2}_{SC})$. According to
Remark \ref{sec:phi_compute}, $s$ can be lifted into a third order solution if
and only if $\varphi(s)=0$, where
\begin{math}
  \varphi(s)=(\tau_1 \nabla P_1(F))_x.
\end{math}
Computing $\varphi(s)$ we find that
\begin{enumerate}
\item using the notation $\omega:=i_Sdd_JF$ we have $\omega_x\!=\!0$ from
  \eqref{eq:P_1_x} and
  \begin{align*}
	\tau^{1}_{S}(\nabla(P_1F) )_x&(X,Y)=\nabla \omega(JX, hY) - \nabla
    \omega(hY, JX) -\nabla \omega(JY, hX) + \nabla \omega(hX, JY)\frac{}{}
	\\
	&=JX\omega(hY)-hY\omega(JX)-JY\omega(hX)+hX\omega(JY)=i_Jd\omega(hX,hY).
  \end{align*}
  Moreover, $di_{S}=-i_Sd+d_S$, $i_Jd_S=i_{[J,S]}+d_Si_J$ and $d_Jd_J=0$, we
  obtain that
  \begin{equation*}
    \begin{aligned}
      i_{J}d\omega(hX,hY)_x&=(i_{J}d_Sdd_{J}F\!-\!i_{J}i_{S}ddd_{J}F)_x(hX,hY)
      \\
      & =(i_{[J,S]}dd_{J}F+d_Si_{J}dd_{J}F)_x(hX,hY)
      =(i_{\Gamma}dd_{J}F)_x(hX,hY).
    \end{aligned}
  \end{equation*}

\item
  \begin{math}
    \tau^{2}_{S}(\nabla (P_1F))_x=(\nabla \omega)_x(X,S)
    =X_x\omega(S)=X_xdd_J(S,S)=0.
  \end{math}
\item Using the identity $J[JX,S]=JX$ we have
  \begin{equation*}
    \begin{aligned}
      \tau^{1}_{SC}(\nabla (P_1F))_x&=X_x(i_{S}dd_{J}F(JY))+X_x(JY(CF-F))
      \\
      &=X_x (-JYd_{J}F(S)-d_{J}F([S,JY]))+X_x(JYCF-JYF)
      \\
      &=-X_x(J[S,JY]F)-X_x(JYF) =X_x(JYF)-X_x(JYF)=0.
    \end{aligned}
  \end{equation*}
\item We have
  \begin{math}
	dd_J(CF-F)(S,hX)=S\big(JX(CF-F)\big)-hX\big(C(CF-F)\big).
  \end{math}
  Then
  \begin{align*}
	\tau^{2}_{SC}(\nabla (P_1F))_x&
    =C(i_Sdd_JF(hX))-S(JX(CF-F))+hX(C(CF-F))
	\\
	&=d_Cdd_JF(S,hX)-dd_Jd_CF(S,hX)+dd_JF(S,hX).
  \end{align*}
  Since  $d_Jd_C-d_Cd_J=d_{[J,C]}=d_J$ it follows that
  \begin{align*}
    d_Cdd_JF-dd_Jd_CF+dd_JF=d_Cdd_JF-dd_Cd_JF=d_Cdd_JF-d_Cdd_JF=0.
  \end{align*}
\end{enumerate}
From the above computation it follows that 
\begin{math}
  \varphi(s)\!=\!(\tau_1 \nabla P_1(F))_x\!=\! (i_{\Gamma}dd_JF_x,0,0,0)
\end{math}
and therefore the only condition to prolong a second order solution into a
third order solution is given by the equation $(i_{\Gamma}dd_{J}F)_{x}=0$ as
Proposition \ref{sec:first-integr-cond} stated.  \hfill \rule{1ex}{1ex}

\bigskip

\begin{proposition}
  \label{thm:P_1_symb_inv}
  The symbol of $P_1=(P_S, P_C)$ is involutive.
\end{proposition}

\begin{proof}
  Let us consider the basis $\mathcal B$ introduced in (\ref{eq:basis_1}).
  Using the notation \eqref{eq:B_notation} we have
  \begin{alignat*}{1}
    g_2(P_1) & =\mathrm{Ker}\, \sigma_2(P_1) =\left\{A\in S^{2}T^{*}|A(X,C)=0,
      \ A(S,JX)=A(X,C)\right\}
    \\
    & =\left\{A\in S^{2}T^{*}| \ A_{ij}\!=\!A_{ji}, \ A_{\u{ij}}\!=\!A_{\u{ji}},
      \quad A_{i \u n}\!=\!A_{n \u i}\!=\!A_{\u{in}}\!=\!A_{\u{ni}}\!=\!0 \right\}
  \end{alignat*}
  and therefore 
  \begin{equation}
    \dim (g_2(P_1))=\frac{n(n+1)}{2}+(n-1)^{2}+\frac{n(n-1)}{2} = n^2+(n-1)^{2}.
  \end{equation}
  Let us consider now the basis $\mathcal E=\{e_i\}_{i=1\dots 2n}$, where
  \begin{equation}
    \label{eq:7}
    \mathcal E    =\Big\{
    \underbrace{h_1}_{e_1}, \dots,\underbrace{h_{n-1}}_{e_{n-1}},
    \underbrace{h_n\!+\!v_1\!+\!\dots\!+\!v_n}_{e_n},
    \underbrace{v_1}_{e_{n+1}},\dots,
    \underbrace{v_n}_{e_{2n}}\Big\}. 
  \end{equation}
  Denoting the coefficients of $A\in S^2T^*$ with respect to $\mathcal E$ by
  $\widetilde A_{ij}$, we have
  \begin{align*}
	g_2&(P_1)_{e_1 \dots e_k}=\left\{A\!\in\! S^{2}T^{*}| i_{e_1}A =0,\dots
      ,i_{e_k}A=0 \right\}
    \\
    & =\left\{A\!\in\! S^{2}T^{*}| \ \tilde A_{ij}\!=\!\tilde A_{ji}, \ \tilde
      A_{\u{ij}}\!=\! \tilde A_{\u{ji}}, \ \tilde A_{i \u n}\!=\!  \tilde A_{n
        \u i}\!=\! \tilde A_{\u{in}}\!=\! \tilde A_{\u{ni}}\!=\!0, \ \tilde
      A_{lj}\!=\!0, \ \tilde A_{l\u{j}}\!=\!0, \ l\leq k \right\}
  \end{align*}
  therefore
  \begin{align*}
	\dim (g_2(P_1))_{e_1\dots e_k}&=
    \left\{ \begin{array}{ll}
        \frac{(n-k)(n-k+1)}{2}+(n-k)(n-1)+\frac{(n-2)(n-1)}{2}, &\textrm{if} \
        k\leq n, \vspace{5pt}
        \\
        \frac{(n-2-k)(n-k-1)}{2}, & \textrm{if} \ k>n,
      \end{array} \right. \vspace{20pt}
  \end{align*}
  and hence 
  \begin{equation*}
    \begin{aligned}
      \dim & g_2(P_1)+\!\!\sum_{k=1}^{2n} \dim g_2 (P_1)_{e_1\dots e_k}
      =n^2\!+\!(n\!-\!1)^{2}
      +\!\!\sum^{n}_{k=1}\Big(\frac{(n\!-\!k)(n\!-\!k\!+\!1)}{2}
      \!+\!(n\!-\!k)(n\!-\!1)\! \Big)
      \\
      & \ +\! \frac{n (n\!-\!2)(n\!-\!1)}{2}
      +\sum^{n}_{k=1}\frac{(n\!-\!2\!-\!k)(n\!-\!k\!-\!1)}{2}
      =\frac{8n^{3}\!-\!9n^{2}\!+\!7n}{6} \ \stackrel{\eqref{eq:6}}{=} \ \dim
      g_3(P_1),
    \end{aligned}
  \end{equation*}
  which shows that the basis \eqref{eq:7} is quasi-regular, and the symbol of
  $P_1$ is involutive.
\end{proof}

\bigskip

\begin{remark}
  Proposition \ref{sec:first-integr-cond} and \ref{thm:P_1_symb_inv} shows
  that the conditions of Theorem \ref{thm:cartan_kahler} are fulfilled if and
  only if for any initial data $j_{2}(F)_x$ of $P_1$ we have also $i_\Gamma
  dd_JF=0$. This is true if $\dim M = 1$. However, when $\dim M \geq 2$, this
  condition does not satisfied by every second order solution, that is not
  every second order solution can be lifted into a third order solution. Since
  the set of initial data is to large (containing some which cannot be
  prolongated into a higher order solution) we have to reduce it by adding the
  compatibility condition into the system. This leads us to consider the
  operator
  \begin{math}
    (P_S,P_C, P_{\Gamma})
  \end{math}
  where $P_\Gamma$ is a second order operator defined as
  \begin{align*}
    P_{\Gamma}:C^{\infty}(TM)\rightarrow \mathrm{Sec}(\Lambda^{2}T^{*}_{v}),
    \qquad P_{\Gamma}F:=i_{\Gamma}dd_JF.
  \end{align*}
\end{remark}

\begin{remark}
  If $S$ is a spray and $F$ is a 1-homogeneous Lagrangian, then we have
  \begin{displaymath}
    P_SF(X)=  i_S dd_JF(X)=dd_JF(S, hX) =\tfrac{1}{2}i_{\Gamma}dd_JF(S,hX)
    =P_\Gamma F(S,hX)
  \end{displaymath}
  for every $X\in T$. Consequently, if $F$ is a solution of $P_\Gamma$ then it
  is also a solution of $P_S$, that is $P_\Gamma$ contains in particular the
  equations of $P_S$. That lead us to drop from the system $P_S$ and consider
  the \emph{extended Rapcsák} system as:
  \begin{equation}
    \label{eq:P_2}
    P_2=(P_\Gamma,  P_C).
  \end{equation}
  It is clear that a function is a solution to the Rapcsák system if and only
  if it is a solution of the extended Rapcsák system.
\end{remark}

\bigskip

\section{Integrability condition of the extended Rapcsák system}
\label{sec:ext_rapcsak}

In this chapter we investigate the integrability of the extended Rapcsák
system $P_2\!=\!(P_{\Gamma}, P_C)$.  Our method is similar to the one we used in
Chapter \ref{sec:rapcsak}.
\begin{proposition}
  \label{thm:second-integr-cond}
  A 2\ts{nd} order solution $s=j_{2}(F)_x$ of the system $P_2=(P_\Gamma, P_C)$
  at $x\in \TM$ can be prolongated into a 3\ts{rd} order solution, if and only
  if $(i_Rdd_JF)_x=0$.
\end{proposition}
\noindent
\emph{Proof.} The symbol of the operator $P_{\Gamma}$ and its first
prolongation are
\begin{alignat*}{2}
  \sigma_{2}(P_{\Gamma})\colon & S^2T^* \rightarrow \Lambda^{2}T^{*}_{v},
  \quad & \big(\sigma_2(P_\Gamma)A_2\big)(Y,Z) &=
  2\big(A_2(hY,JZ)\!-\!A_2(hZ,JY)\big),
  \\
  \sigma_{3}(P_{\Gamma})\colon & S^3T^* \rightarrow T^{*}\!\!\otimes\!
  \Lambda^{2}T^{*}_{v}, \ \ & \big(\sigma_3(P_\Gamma)A_3\big)(X,Y,Z) &=
  2\big(A_3(X,hY,JZ)\!-\!A_3(X,hZ,JY)\big),
\end{alignat*}
where $X,Y,Z \!\in\! T$, $A_2\!\in\! S^{2}T^{*}$, $A_3\!\in\!  S^3T^{*}$. Let
us consider the map
\begin{align} 
  \label{tau2}
  \tau_2 := (\tau^{1}_{\Gamma}, \tau^{2}_{\Gamma},\, \tau_{\Gamma C})
\end{align}
defined on $(T^{*}\!\!\otimes \! \Lambda^{2}T^{*}_{v})\!\times S^2T^{*}$ with
\begin{alignat}{2}
  \label{eq19}
  &\tau^{1}_{\Gamma}(B_{\Gamma},B_C)(X,Y,Z)&&=B_{\Gamma}(hX, Y, Z) +
  B_{\Gamma}(hY, Z, X) +B_{\Gamma}(hZ, X, Y),
  \\
  \label{eq20}
  &\tau^{2}_{\Gamma} (B_{\Gamma},B_C)(X,Y,Z)&&=B_{\Gamma}(JX, Y, Z)
  +B_{\Gamma}(JY, Z, X) +B_{\Gamma}(JZ, X, Y),
  \\
  \label{eq:21}
  &\tau_{\Gamma C} (B_{\Gamma},B_C)(X,Y)&&=\tfrac{1}{2}B_\Gamma
  (C,X,Y)-B_C(hX,JY)+B_C(hY,JX),
\end{alignat}
where $B_{\Gamma}\!\in\! T^{*} \!\!\otimes \! \Lambda^{2}T^{*}_{v}$, $B_C
\!\in\!  S^2T^{*}$, $X,Y,Z \!\in\! T$.  We have the following

\medskip

\begin{lemma}
  \label{sec:lemma_2} 
  Let $K_2$ be the image of $\tau_2$. Then the sequence
  \begin{align}
    \label{exact2}
    S^{3}T^{*}\xrightarrow{ \ \sigma_3(P_2) \ } (T^{*}\!\otimes
    \Lambda^{2}T^{*}_{v})\times S^2T^*\xrightarrow{ \ \tau_2 \ }
    K_2 \longrightarrow 0
  \end{align}
  is exact.  
\end{lemma}
\begin{proof}
  A simple computation shows that $\sigma_3(P_2) \circ \tau_2 = 0$, and
  therefore $\Im \sigma_3(P_2)\!\subset\! \Ker \tau_2$. Let us compute the
  rank of $\sigma_3(P_2)$.  By using the basis \eqref{eq:basis_1} and the
  notation \eqref{eq:B_notation}, a symmetric tensor $A\in S^{3}T^{*}$ is an
  element of $\mathrm{Ker}\,\sigma_3(P_2)$ if in addition of the relations
  describing the symmetry properties, the equations \eqref{eq:symb_1} and the
  equations
  \begin{equation}
    \label{eq:317_318}
	A_{ij \u k}=A_{ik \u j}, \qquad
    A_{\u i j \u k}=A_{\u i k \u j}, \qquad i,j,k=1,\dots ,n,
  \end{equation}
  hold.  We obtain from \eqref{eq:317_318} that all of the blocks
  \eqref{eq:B_notation} are totally symmetric, and
  $A_{ij\u{n}}=A_{i\u{jn}}=A_{\u{ijn}}=0$. That way there are
  $\frac{n(n+1)(n+2)}{6}$ free components in the block $A_{ijk}$ and
  $\frac{(n-1)n(n+1)}{6}$ free components to choose in each of the blocks
  $A_{ij\u{k}}$, $A_{i\u{jk}}$ and $A_{\u{ijk}}$.  That is
  \begin{equation}
    \label{eq:g_3_P_2}
    \dim(g_3(P_2))=\frac{n(n+1)(n+2)}{6} + 3\frac{(n-1)n(n+1)}{6}
    =\frac{4n^{3}+3n^{2}-n}{6},
  \end{equation}
  and 
  \begin{equation}
    \label{eq:4}
    \rank \sigma_3(P_2) = \dim S^{3}T^{*} -\dim(g_3(P_2))
    =\frac{4n^{3}+9n^{2}+5n}{6}.
  \end{equation}
  On the other hand, considering the equations of $\mathrm{Ker}\,\tau_2$, we
  can find that the pivot terms for $\tau^{1}_{\Gamma}\!=\!0$ and for
  $\tau^{2}_{\Gamma}\!=\!0$ are $B_{\Gamma}(h_i,h_j,h_k)$ and
  $B_{\Gamma}(v_i,h_j,h_k)$, $i\!<\!j\!<\!k$, respectively. Therefore each of
  them give $\binom{n}{3}$ independent equations.  Furthermore
  $B_{\Gamma}(v_n,h_i,h_j)$, $i<j$, $i,j\!=\!1\dots n$, are pivot terms for
  $\tau_{\Gamma C}=0$ which gives $\frac{n(n-1)}{2}$ independent equations.
  Hence
  \begin{equation}
    \label{eq:8}
	\dim \mathrm{Ker} \, \tau_2=\dim S^2T^* \!\! + \dim(T^{*}\! \otimes
    \Lambda^{2}T^{*}_{v})\!-\!2\binom{n}{3}
    \!-\!\frac{n(n\!-\!1)}{2}=\frac{4n^{3}\!+\!9n^{2}\!+\!5n}{6}.
  \end{equation}
  Comparing the dimensions \eqref{eq:4} and \eqref{eq:8} one can find that
  \begin{math}
    \rank \sigma_3(P_2)= \dim \Ker \, \tau_2
  \end{math}
  and the sequence (\ref{exact2}) is exact.
\end{proof}

\bigskip 

\noindent
Let us turn our attention to the proof of Proposition
\ref{thm:second-integr-cond}. We have the following commutative diagram:
\begin{displaymath}
  \begin{CD}
    g_3(P_2)@>>> S^{3}T^{*} @>\sigma_3(P_2)>> (T^{*}\!\otimes \!
    \Lambda^{2}T^{*}_{v})\!\times S^2T^{*}@>\tau_2>> K_2 \longrightarrow 0
    \\
    @VVV @VV \epsilon V @VV\epsilon V
    \\
    Sol_3(P_2) @> i >> J_3 (\RTM) @>p_3(P_2)>> J_1(\Lambda^{2}T^{*}_{v})
    \! \times\! J_2 (\RTM)
    \\
    @VV \overline{\pi}_2V @VV \pi_2 V @VVV
    \\
    Sol_2(P_2) @> i >> J_2(\RTM) @>p_2(P_2)>> \Lambda^{2}T^{*}_{v}\!\times\! J_1
    (\RTM)
  \end{CD}
\end{displaymath}
Let $s\!=\!j_{2}(F)_x$ be a second order solution of $P_2$ at a point $x$,
that is $(P_2F)_x=0$. We have
\begin{equation}
  \label{eq:P_2_x}
  (i_\Gamma dd_JF)_x\!=\!0, \qquad (\L_CF\!-\!F)_x\!=\!0, \qquad 
  \big(\nabla(\L_CF\!-\!F)\big)_x\!=\!0.
\end{equation}
The integrability condition can be computed with the help of the map $\tau_2$
(see Proposition \ref{sec:phi} and \ref{sec:phi_compute}). Indeed,
$s\!\in\!Sol_{2,x}(P_2)$ can be prolongated into a third order solution if and
only if $\varphi(s)=0$, where
\begin{math}
  \varphi(s)=(\tau_2 \nabla P_2(F))_x.
\end{math}
Let us introduce the notation $\Omega=dd_JF$.  Using the component maps of
$\tau_2$ introduced in \eqref{tau2} one can find
\begin{enumerate}
\item
  \begin{math}
    \begin{aligned}[t]
      \displaystyle \tau^{1}_{\Gamma}(\nabla(P_{2}F))_x&
      =d_{h}(i_{\Gamma}dd_JF)_x
      =(d_hi_{2h-I}dd_JF)_x =(2(d_hi_hdd_JF-d_hdd_JF))_x=
      \\
      &=(2d_h(i_hd-di_h)d_JF)_x=(2d_hd_hd_JF)_x\frac{}{}
      =(d_Rd_JF)_x=(i_R\Omega)_x.
    \end{aligned}
  \end{math}
\item 
  \begin{math}
    \begin{aligned}[t]
      \tau^{2}_{\Gamma} (\nabla & (P_{2}F))_x= d_{J}(i_{\Gamma} \Omega)_x
      \!\stackrel{\eqref{eq:h_v}}{=}\!
      (d_J(i_{2h-I}\Omega))_x\!=\!(2d_Ji_hdd_JF-2d_Jdd_JF)_x
      \\
      &\!=\!(-2d_Ji_hd_JdF \!-\!2i_Jddd_JF\!+\!2di_Jdd_JF)_x
      \!=\!-(2d_J(d_Ji_hdF+d_JdF))_x\!=\!0
    \end{aligned}
    \medskip
  \end{math}
  \\
  where we used $[d,d_J]=0$, $[i_h,d_J]=d_{Jh}-i_{[h,J]}$ and $[J,h]=0$.
\item
  \begin{math}
    \begin{aligned}[t]
      \tau_{\Gamma C}(\nabla & P_2(F))_x(X,Y)
      =\tfrac{1}{2}\nabla i_{\Gamma}\Omega(C,X,Y)-\nabla P_CF(hX,JY)+\nabla
      P_CF(hY,JX)
      \\
      &=\tfrac{1}{2}d_Ci_{\Gamma}\Omega(hX,hY)
      -\tfrac{1}{2}i_{\Gamma}d_Cdd_JF(hX,hY)
      =\tfrac{1}{2}d_{[C,\Gamma]}\Omega(hX,hY) \stackrel{[C,\Gamma]=0}{=}0.
    \end{aligned}
  \end{math}
\end{enumerate}
The above computation shows that 
\begin{math}
  \varphi(s)\!=\!\tau_2(\nabla P_2(F))_x\!=\!(i_R\Omega_x,0,0)
\end{math}
which proves Proposition \ref{thm:second-integr-cond}.  \hfill \rule{1ex}{1ex}

\bigskip

\begin{proposition}
\label{lemma_2}
  The symbol of $P_2$ is involutive.
\end{proposition}

\begin{proof} 
  We consider the basis \eqref{eq:basis_1} and use the notation
  \eqref{eq:B_notation}.  We have
  \begin{alignat*}{1}
    g_2(P_2) =\mathrm{Ker}\, \sigma_2(P_1) & = \left\{A\in
      S^{2}T^{*}|A(X,C)=0, \ A(hX,JY)=A(hY,JX)\right\}
    \\
    & =\left\{A\in S^{2}T^{*}| \ A_{ij}\!=\!A_{ji}, \
      A_{i\u{j}}\!=\!A_{j\u{i}},\ A_{\u{ij}}\!=\!A_{\u{ji}}, \ A_{i \u
        n}\!=\!0, \ A_{n \u i}\!=\!0 \right\}.
  \end{alignat*}
  Therefore
  \begin{math}
    \dim (g_2(P_2)) =\frac{n(n+1)}{2}+2\frac{(n-1)n}{2}.
  \end{math}
  Let us consider the basis $\widehat {\mathcal E}=\{\widehat
  e_i\}_{i=1}^{2n}$, where
  \begin{alignat*}{2}
    \widehat{e}_i&=h_i+iv_i,& &i=1,\dots, n-1,
    \\
    \widehat{e}_n&=h_n+v_1+\dots+v_n, \quad & &
    \\
    \widehat{e}_{i+n}&=v_i,& &i=1,\dots, n.
  \end{alignat*}
  In the new basis the components of the block $\hat{A}_{\u{ij}}=A_{\u{ij}}$
  can be expressed as a combination of the components $\hat{A}_{i\u j}$ as
  follows: when $i\!\neq\! j$, then 
  \begin{math}
    \hat A_{\u{ij}}\!=\!\frac{1}{i\!-\!j}(\hat{A}_{i \u j}\!-\!\hat{A}_{j
      \u i}),
  \end{math}
  and if $i\!=\! j$ we have
  \begin{math}
    \hat A_{\u{jj}}\!=\!\hat{A}_{n\u j}\!-\!\sum_{k\neq j}
    \frac{1}{k\!-\!j}(\hat{A}_{j\u k}\!-\!\hat{A}_{k \u j}).
  \end{math}
  Then,
  \begin{align*}
	\dim (g_2(P_2))_{\widehat{e}_1 \dots \widehat{e}_k}&= \left\{
      \begin{array}{lll} 
        \frac{(n-k+1)(n-k)}{2} +(n-1)(n-k), \quad &
        \textrm{if} \ k\leq n,
        \\
        \ 0, & \textrm{if} \ k>n,
      \end{array} \right.
  \end{align*}
  and therefore
  \begin{alignat*}{1}
    \dim & (g_2(P_2))+\sum_{k=1}^{2n} \dim
    (g_2(P_2))_{\widehat{e}_1\dots\widehat{e}_k}=
    \\
    &=\frac{n(n\!+\!1)}{2} \!+\!2\frac{(n\!-\!1)n}{2}
    \!+\!\sum^{n}_{k=1}\left(\frac{(n\!-\!k\!+\!1)(n\!-\!k)}{2}+(n\!-\!1)(n\!-\!k)
    \right) \stackrel{\eqref{eq:g_3_P_2}}{=}\dim(g_3(P_2)),
  \end{alignat*}
  and the basis $\widehat {\mathcal E}$ is quasi-regular.
\end{proof}

\bigskip

\noindent

From Proposition \ref{thm:second-integr-cond} and \ref{lemma_2}, using the
Cartan--K{\"a}hler theorem, we get that the integrability condition of the
extended Rapcsák equation can be given in terms of the curvature tensor and we
have the following

\begin{theorem} 
  Let $S$ be a spray on a manifold $M$. If \vspace{-5pt}
  \begin{enumerate}
  \item $\dim M=2$, \vspace{-5pt}
  \item the spray $S$ is flat, \vspace{-5pt}
  \item the spray $S$ is of isotropic curvature, \vspace{-5pt}
  \end{enumerate}
  then the extended Rapcsák equation is formally integrable.
\end{theorem}

\begin{proof}
  To prove the formal integrability one has to show that
  \begin{math}
    \overline{\pi}_l\colon Sol_{l+1}(P)\to Sol_{l}(P)
  \end{math}
  maps are surjective for any $l\geq 2$. Let $s=j_{2}(F)_x$ be a second order
  solution of the system $(P_\Gamma, P_C)$ at $x\in \mathcal{T}M$. According
  to Proposition \ref{thm:second-integr-cond} it can be prolongated into a
  3\ts{rd} order solution if and only if $(i_Rdd_JF)_x=0$ holds.
  \begin{enumerate}
  \item If $\dim M=2$, then the space of semi-basic 3-forms is trivial, that
    is $\Lambda^3_v(TM)=\{0\}$. Therefore $i_Rdd_JF=0$.
  \item If $S$ is flat, that is $\Phi=\lambda J$, then $R=d_{J} \lambda \wedge
    J$. Using the integrability of the vertical distribution we get:
    \begin{math}
      i_{R}dd_{J}F=d_{R}d_{J}F
      =d_{J}\lambda d^{2}_{J}F+d_{d_{J}\lambda}\wedge i_{J}d_{J}F=0.
    \end{math}
  \item If $S$ is of isotropic curvature, then $R$ takes the form
    $R\!=\!\alpha \wedge J + \beta \otimes C$, where $\alpha \!\in\!
    \Lambda^1_v(TM)$, $\beta \!\in\!  \Lambda^2_v(TM)$. Then
    \begin{displaymath}
      i_Rdd_JF=i_{\alpha\wedge J+\beta\otimes C}dd_JF
      = \alpha\wedge i_Jdd_JF+\beta \wedge i_Cdd_JF=0.
    \end{displaymath}
  \end{enumerate}
  The above computation shows that in the above cases all 2\ts{nd} order
  solutions can be prolongated into a 3\ts{rd} order solution. Moreover, as
  Theorem \ref{sec:lemma_2} shows, the symbol is involutive. Therefore,
  according Cartan--K{\"a}hler theorem, the operator $P_2$ is formally
  integrable.
\end{proof}

\begin{corollary} 
  Let $S$ be a analytic spray on an analytic manifold $M$. If $M$ is
  2-dimensional, or the spray $S$ is flat, resp.~ of isotropic curvature, then
  $S$ is locally projectively Finsler metrizable.
\end{corollary}
Indeed, in the analytic context, formal integrability implies the existence of
solutions for all the initial data.

\bigskip 

The integrability condition $i_{R}dd_JF=0$ also appeared in \cite{Bacso-Z}.
It can be shown (c.f.~\cite{CMS_2012}) that this integrability condition is
equivalent to the equation $i_\Phi dd_JF=0$ or $i_W dd_JF=0$, where $\Phi$ is
the Jacobi endomorphism and $W$ is the Weyl tensor associated to $S$.

\end{document}